\documentclass[oneside, 12pt]{amsart}

\usepackage[utf8]{inputenc}
\usepackage[margin=1in]{geometry}
\usepackage{amsthm}
\usepackage{enumitem}  
\usepackage[OT2,T1]{fontenc}
\usepackage{amscd, amssymb, enumitem, amsmath,mathrsfs, amsfonts, amsaddr}
\usepackage{url}
\usepackage{tikz}
\usepackage{fullpage}
\usepackage{microtype}
\usetikzlibrary{decorations.pathreplacing}
\usepackage[backref=page]{hyperref}
\usepackage{hyperref}
\usepackage[utf8]{inputenc}
\usepackage[main=english,russian]{babel}
\usepackage{comment}

\DeclareSymbolFont{cyrletters}{OT2}{wncyr}{m}{n}
\DeclareMathSymbol{\Sha}{\mathalpha}{cyrletters}{"58}

\usepackage[utf8]{inputenc}

\newtheorem{theorem}{Theorem}[section]
\newtheorem{lemma}[theorem]{Lemma}
\newtheorem{proposition}[theorem]{Proposition}
\newtheorem*{proposition*}{Proposition}

\newtheorem*{questiona*}{Question A}
\newtheorem*{questionb*}{Question B}
\newtheorem*{theorem*}{Theorem}
\newtheorem*{question*}{Question}

\theoremstyle{definition}

\newtheorem{remark}[theorem]{Remark}

\newtheorem*{acknowledgement}{Acknowledgement}

\theoremstyle{remark}

\title{Small Tamagawa numbers of elliptic curves with isogenies or torsion}
\author{Mentzelos Melistas}

\address{University of Twente, Department of Applied Mathematics, Drienerlolaan 5, 7522 NB Enschede, The Netherlands}

\begin{document}

\maketitle
\begin{abstract}
    In this article we study Tamagawa numbers of elliptic curves defined over $\mathbb{Q}$ that have  isogenies or torsion points. More precisely, our aim is either to bound the number of primes that can divide their Tamagawa numbers or, when such a bound is not possible, to find infinite subfamilies whose Tamagawa numbers are as small as possible. Finally, we also investigate Tamagawa numbers of specializations of elliptic surfaces.
\end{abstract}

\section{Introduction}

Let $E/\mathbb{Q}$ be an elliptic curve and $p$ be a prime number. Let $E_0(\mathbb{Q}_p)$ be the subgroup of $E(\mathbb{Q}_p)$ consisting of points with non-singular reduction. It is well known, see e.g. \cite[Section VII.6]{aec}, that $E_0(\mathbb{Q}_p)$ has finite index in $E(\mathbb{Q}_p)$. We call $c_{p}(E)=[E(\mathbb{Q}_p):E_0(\mathbb{Q}_p)]$ the Tamagawa number of $E/\mathbb{Q}$ at $p$. Alternatively, one can define the Tamagawa number of $E/\mathbb{Q}$ at $p$ as $|\mathcal{E}_{\mathbb{F}_p}(\mathbb{F}_p)/ \mathcal{E}(\mathbb{F}_p)|$, where $\mathcal{E}_{\mathbb{F}_p}/\mathbb{F}_p$ is the special fiber of the Néron model of $E/\mathbb{Q}$ at $p$
and $\mathcal{E}_{\mathbb{F}_p}/\mathbb{F}_p$ is the connected component of the identity of $\mathcal{E}_{\mathbb{F}_p}/\mathbb{F}_p$ (see \cite[Corollary IV.9.2]{silverman2} for a proof that these two definitions are equivalent). 

We define the (global) Tamagawa number of $E/\mathbb{Q}$ as $c(E):=\prod_{p} c_{p}(E)$, where the product is taken over all the primes of bad reduction of $E/\mathbb{Q}$. Lorenzini in \cite{lor} studied possible cancellations in the quotient $c(E)/|E(\mathbb{Q})_{\mathrm{tors}}|$, where $E(\mathbb{Q})_{\mathrm{tors}}$ is the torsion subgroup of $E/\mathbb{Q}$. As motivation behind such an investigation, we note that the fraction $c(E)/|E(\mathbb{Q})_{\mathrm{tors}}|$ appears in the Birch and Swinnerton-Dyer Conjecture \cite[Appendix C.16]{aec}.

Among other results, Lorenzini in \cite{lor} proved that for every prime $N \geq 5$ if $E/\mathbb{Q}$ is an elliptic curve with a $\mathbb{Q}$-rational point of order $N$, then $N $ divides $c(E)$, with only finitely many exceptions. Results of similar flavor for elliptic curves over number fields have been proved by Krumm, Najman, and the current author in \cite{Krummthesis}, \cite{najmantamawanumber}, and \cite{mentzelosarch}, respectively. On the other hand, Tamagawa numbers of elliptic curves with isogenies have been studied by Barrios and Cullinan, and by Trbović in \cite{barrioscullinantamagawanumbersellipticcurves} and \cite{trbovic}, respectively. In this article, we study the divisibility properties of Tamagawa numbers of elliptic curves that have either isogenies or torsion points. We focus here on two directions. First, on excluding primes that divide Tamagawa numbers of elliptic curves with isogenies. Second, on finding infinite families of elliptic curves with Tamagawa numbers which are as small as possible.

We first study Tamagawa numbers of elliptic curves over $\mathbb{Q}$ that have cyclic $\mathbb{Q}$-rational isogenies. The possible degrees of $\mathbb{Q}$-rational isogenies of elliptic curves defined over $\mathbb{Q}$ have been classified by Mazur, Kenku, and others (see \cite[Theorem 1]{mazur} and \cite[Theorem 1]{kenku} for more information). Namely, if $E/\mathbb{Q}$ is an elliptic curve with a $\mathbb{Q}$-rational isogeny of degree $N$, then $N \leq 19$, or $N= 21, 25, 27, 37, 43, 67, 163$. Our theorem below concerns how the existence of such isogenies imposes restrictions on the primes that can divide their Tamagawa numbers.

\begin{theorem}\label{theoremisogeniesalldegrees}
    Let $E/\mathbb{Q}$ be an elliptic curve with a cyclic $\mathbb{Q}$-rational isogeny of degree $N$.
    \begin{enumerate}
        \item If $N=14, 19, 43, 67,$ or $163$, then $c(E)=2^n$, for some $n \geq 1$.
        \item If $N=11,27, 37$, then $c(E)=2^n3^m$, for some $n \geq 0$ and $m \in \{0,1\}$.
        \item If $N=17$, then $c(E)=2^n3^m17^k$, for some $n \geq 1$ and $m,k \in \{0,1\}$.
        \item If $N=21$, then $c(E)=2^n3^m7^k$, for some $n \geq 1$, $m \in \{0,1,2 \}$, and $k \in \{0,1\}$.
        \item If $\ell=3$ or $5$ and $N=3\ell$, then $c(E)=2^n3^m\ell^k$, for some $n \geq 1$, $m \in \{0,1,2 \}$, and $k \in \{0,1\}$.
    \end{enumerate}
\end{theorem}

    Let $N'=4,5,6,7,8,9,10$ or $12$. As we explain in Proposition \ref{propositiondivisors}, for every prime $p$ there exist infinitely many elliptic curves $E/\mathbb{Q}$ with a $\mathbb{Q}$-rational point of order $N'$ (and hence with a cyclic $\mathbb{Q}$-rational isogeny of order $N'$) such that $p$ divides $c(E)$. Thus, the above theorem cannot be extended to elliptic curves with isogenies of degree $N'=4,5,6,7,8,9,10$ or $12$. On the other hand, we show  that for each $\ell=5,7,$ or $13$, there exist infinitely many elliptic curves $E/\mathbb{Q}$ (with distinct $j$-invariants) that have a $\mathbb{Q}$-rational isogeny of degree $\ell$ such that $c(E)=2^n3^m$, for some $n,m$ (see Proposition \ref{prop5713} below).

We now turn our attention to elliptic curves with torsion points. Since it is not possible to bound the primes that can divide their Tamagawa numbers, the next best question to ask is whether we can always find infinite families of elliptic curves with torsion points and Tamagawa numbers as small as possible. As a step towards this direction we offer the following theorem.

\begin{theorem}\label{theoreminfinitefamiliestorsion}
    \begin{enumerate}
        \item There exist infinitely many elliptic curves $E/\mathbb{Q}$ with a $\mathbb{Q}$-rational point of order $4$ and such that $c(E) =4, 8,$ or $12$.
        \item There exist infinitely many elliptic curves $E/\mathbb{Q}$ with a $\mathbb{Q}$-rational point of order $5$ such that $c(E) \leq 30$.
    \end{enumerate}
\end{theorem}

Inspired by the methods used to prove Theorem \ref{theoreminfinitefamiliestorsion}, we also propose a general question concerning Tamagawa numbers of specializations of families of elliptic curves. Let $E/\mathbb{Q}(T)$ be an elliptic curve. Then, for almost all values $t \in \mathbb{Q}$ the specialization $E_t/\mathbb{Q}$ of $E/\mathbb{Q}(T)$ at $T=t$ is an elliptic curve. More precisely, if $E/\mathbb{Q}(T)$ is given by a Weierstrass equation of the form $$y^2+a_1(T)xy+a_3(T)y=x^3+a_2(T)x^2+a_4(T)x+a_6(T),$$ with $a_i(T) \in \mathbb{Q}(T)$ for $i=1,...,6$, then $E_t/\mathbb{Q}$ will be an elliptic curve provided that $a_i(t) \neq \infty $ for all $i=1,...,6$, and that the discriminant $\Delta(t) \neq 0$.  Therefore, one can think of $E/\mathbb{Q}(T)$ as an one-parameter family of elliptic curves. We say that $E/\mathbb{Q}(T)$ is non-isotrivial if $j(E/\mathbb{Q}(T)) \not\in \mathbb{Q}$.

Consider now
$$\mathcal{C}_1=\min \{ c(E_t) \: | \: t \in \mathbb{Q} \text{ and } E_t/\mathbb{Q} \text{ is an elliptic curve} \}   $$
and 
$$\mathcal{C}_2=\min \{ c \: | \: \exists \text{ infinitely many } t \in \mathbb{Q}  \text{ such that  } E_t/\mathbb{Q} \text{ is an elliptic curve with } c(E_t)=c \} .$$

One might hope that $\mathcal{C}_1=\mathcal{C}_2$. However, this is not true in general. For example, for the family consider in the proof of Proposition \ref{propfivetorsion} below, we have that $\mathcal{C}_1=1$ while $\mathcal{C}_2 \geq 5$ (see \cite[Proposition 2.7]{lor}). On the other hand, it can happen that $\mathcal{C}_1=\mathcal{C}_2$ for some families. Indeed, there exist infinitely many elliptic curves $E/\mathbb{Q}$ which belong to one-parameter families of elliptic curves with $\mathbb{Q}$-rational points of order $3$, respectively $2$, such that $c(E)=1$, see \cite[Lemma 2.26]{lor}, respectively \cite[Corollary 5.3]{barriosroy}. 

It seems likely that the difference between $\mathcal{C}_1$ and $\mathcal{C}_2$ can be controlled in a way that depends only on the discriminant of $E/\mathbb{Q}(T)$. We propose this in the following question.

\begin{questiona*}
    Let $E/\mathbb{Q}(T)$ be a non-isotrivial elliptic curve and let $\Delta(T)$ be the discriminant of a Weierstrass equation for $E/\mathbb{Q}(T)$ which is minimal at all places of $\mathbb{Q}[T]$. Does the exist a constant $d$, that depends only on $\mathrm{deg}( \Delta)$, such that $$\mathcal{C}_2 \leq \mathcal{C}_1+d?$$
\end{questiona*}

Our theorem below provides some evidence that the previous question might have a positive answer. Before we state our result, let us fix the following notation; if $F(x)$ is a polynomial in $\mathbb{Z}[x]$, then we will write $\rho_F(p)$ for the number of solutions of the congruence $$F(x) \equiv 0 \: ( \text{mod } p).$$ 

 \begin{theorem}\label{generalthm}
    Let $E/\mathbb{Q}(T)$ be a non-isotrivial elliptic curve and let $\Delta(T)$ be the discriminant of a Weierstrass equation for $E/\mathbb{Q}(T)$ which is minimal at all places of $\mathbb{Q}[T]$. Write $\Delta(T)=mF_1^{m_1}(T) \cdots F_g^{m_g}(T)$, where $F_1(x), F_2(x),...,F_g(x) $ be distinct irreducible polynomials with integral coefficients and $m \in \mathbb{Z}$. Let $F(x)=F_1(x)F_2(x) \cdots F_g(x)$ and assume that $\rho_F(p)<p$ for every prime $p$. Then there exists a positive integer $s$ that can be explicitly computed and depends on $\Delta(T)$ such that there exist infinitely many $n \in \mathbb{N}$ with $$c(E_n) \leq 16 (\log_2(m)+ s \: \mathrm{deg}(\Delta))^{d(m)+s},$$ where $d(m)$ is the number of primes that divide $m$. In particular, we have that $$\mathcal{C}_2 \leq \mathcal{C}_1+16 (\log_2(m)+ s \: \mathrm{deg}(\Delta))^{d(m)+s}.$$
\end{theorem}

When $\Delta(T)$ is either a quadratic or cubic polynomial, under some (mild) assumptions, we prove that the term $16 (\log_2(m)+ s \: \mathrm{deg}(\Delta))^{d(m)+s}$ in Theorem \ref{generalthm} can be improved to $16 (\log_2(m)+1)^{d(m)}$ (see Theorem \ref{theoremspecialization}). Finally, relying on the abc-conjecture we show that, assuming that the discriminant $\Delta(T)$ of a Weierstrass equation which is minimal at all places of $\mathbb{Q}[T]$ has no repeated roots and that the largest integer that divides all elements of the set $\{ \Delta(s) \: | \: s \in \mathbb{Z}\}$ is square-free, there exist infinitely many specializations with Tamagawa number $1$.

\begin{acknowledgement}
    The author would like to thank Dino Lorenzini for some very helpful comments on an earlier version of this manuscript.
\end{acknowledgement}

\section{Tamagawa Numbers and Isogenies}

Let $N$ be a positive integer and let $E/\mathbb{Q}$ be an elliptic curve with a $\mathbb{Q}$-rational isogeny of degree $N$. Let $X_0(N)$ be the modular curve parametrizing elliptic curves with an isogeny of degree $N$. We refer the reader to \cite{shimurabook} or \cite{katzmazur} for general background on modular curves. In \cite[Section 3]{mentzelosisogenies}, using explicit equations for $X_0(N)$ the possible reduction types of $E/\mathbb{Q}$ when $N$ is equal to $11, 17, 19, 37, 43, 67,$ or $163$ where computed. Combining this technique with with known results connecting reduction types and Tamagawa numbers, in this section, we prove Theorem \ref{theoremisogeniesalldegrees}. We then consider the case where $N=5,7,$ or $13$ and we prove that there exist infinitely many elliptic curves $E/\mathbb{Q}$ (with distinct $j$-invariants) that have a $\mathbb{Q}$-rational isogeny of degree $\ell$ such that $c(E)=2^n3^m$, for some $n,m$. Finally, for elliptic curves with a torsion point of order $7$ we prove a result that relates their rank with the $7$-adic valuation of their Tamagawa numbers (assume finiteness of all relevant Tate-Shafarevich groups).

      In Table 1 below we collect several important observations that come from an algorithm of Tate \cite{tatealgorithm}, and which will be used repeatedly. The reader is referred to \cite[Section IV.9]{silverman2} for more information (and proofs). 

      \begin{table}[ht]\label{table1}
       \begin{center}
    \begin{tabular}{ ||c|c || } 
    \hline \hline
    Reduction type of $E/\mathbb{Q}$ modulo $p$ & Tamagawa number of $E/\mathbb{Q}$ modulo $p$ \\ 
    \hline
    \hline
     \textup{I}$_0$ & 1 \\ 
    split \textup{I}$_n$ & $n$  \\
    non-split \textup{I}$_n$ with $n$ even & $2$  \\
    non-split \textup{I}$_n$ with $n$ odd & $1$  \\
    \textup{II} & $1$ \\
    \textup{III} & $2$  \\
    \textup{IV} & $1$ or $3$  \\
    \textup{I}$_n^*$  & $2$ or $4$  \\
    \textup{II}$^*$  & $1$  \\
    \textup{III}$^*$  & $2$  \\
    \textup{IV}$^*$  & $1$ or $3$  \\
    \hline
   \end{tabular}
   \caption{Reduction types and Tamagawa numbers}
   \label{table}
\end{center}
\end{table}

We are now ready to proceed to our results.

      \begin{proposition}\label{manyell}
          Let $\ell$ be equal to $11, 19, 43, 67,$ or $163$, and let $E/\mathbb{Q}$ be an elliptic curve with a $\mathbb{Q}$-rational isogeny of degree $\ell$. 
          \begin{enumerate}
              \item If $p \neq 2, \ell$ is a prime, then $c_p(E)=1,2,$ or $4$.
              \item If $\ell=11$, then $c_{\ell}(E)=1,2, $ or $3$.
              \item If $\ell= 19, 43, 67,$ or $163$, then $c_{\ell}(E)=2$
              \item We have that $c_{2}(E)=1,2,$ or $4$.
          \end{enumerate}
          In particular, if $\ell=11$, then $c(E)=2^n3^m$, for some $n \geq 0$ and $m \in \{0,1\}$. Moreover, if $\ell= 19, 43, 67,$ or $163$, then $c(E)=2^n$, for some $n \geq 1$.
      \end{proposition}
      \begin{proof}
          Proof of $(i)$: Let $p \neq 2,\ell$ be a prime. According to \cite[Theorems 3.3 \& 3.4]{mentzelosisogenies} the curve $E/\mathbb{Q}$ has either good reduction or reduction of type \textup{I}$_0^*$ modulo $p$. Therefore, using Table \ref{table} we see that $c_p(E)=1,2$ or $4$. 

           Proof of $(ii)$: It follows from \cite[Theorem 3.4, Part $(ii)$]{mentzelosisogenies} that $E/\mathbb{Q}$ has reduction of type  \textup{II}, \textup{II}$^*$, \textup{III}, \textup{III}$^*$, \textup{IV}, or \textup{IV}$^*$ modulo $11$. Therefore, using Table \ref{table} we see that $c_{11}(E)=1,2, $ or $3$.

           Proof of $(iii)$: Using \cite[Theorem 3.3, part $(ii)$]{mentzelosisogenies},  we obtain that $E/\mathbb{Q}$ has reduction of type \textup{III} or \textup{III}$^*$ modulo $\ell$. Combining this with Table \ref{table} we get that

           Proof of $(iv)$: From \cite[Theorems 3.3 \& 3.4]{mentzelosisogenies}, we see that $E/\mathbb{Q}$ has either good reduction or reduction of type \textup{I}$_4^*$, \textup{I}$_8^*$, \textup{II}, or \textup{II}$^*$ modulo $2$. Thus, using Table \ref{table} we find that $c_{2}(E)=1,2,$ or $4$.

           The last part of the proposition is immediate since $c(E)=\prod_p c_p(E)$ and $c_p(E)=1$ for all but finitely many $p$.
      \end{proof}

      \begin{remark}
           The curves $E_1/\mathbb{Q}$, $E_2/\mathbb{Q}$, and $E_3/\mathbb{Q}$ with LMFDB \cite{lmfdb} labels\href{https://www.lmfdb.org/EllipticCurve/Q/121/a/2}{121.a2}, \href{https://www.lmfdb.org/EllipticCurve/Q/121/b/1}{121.b1}, \href{https://www.lmfdb.org/EllipticCurve/Q/20449/c/1}{20449.c1}, respectively, have a $\mathbb{Q}$-rational isogeny of degree $11$ and $c(E_1)=1$, $c(E_2)=2$, $c(E_3)=3$. 
      \end{remark}

      The proofs of the two propositions below are entirely analogous to the proof of Proposition \ref{manyell}, using \cite[Theorem 3.6 and Theorem 3.7]{mentzelosisogenies} combined with Table \ref{table} above. We will omit the details of the proofs.

      \begin{proposition}\label{propdegree17}
         Let $E/\mathbb{Q}$ be an elliptic curve with a $\mathbb{Q}$-rational isogeny of degree $17$. 
          \begin{enumerate}
              \item If $p \neq 2, 5, 17$ is a prime, then $c_p(E)=1,2,$ or $4$.
              \item The curve $E/\mathbb{Q}$ has $c_2(E)=1,2,4,$ or $17$.
              \item The curve $E/\mathbb{Q}$ has $c_5(E)=2$.
              \item The curve $E/\mathbb{Q}$ has $c_{17}(E)=1$ or $3$.
          \end{enumerate}
          In particular, $c(E)=2^n3^m17^k$, for some $n \geq 1$ and $m,k \in \{0,1\}$.
      \end{proposition}

      \begin{proposition}
          Let $E/\mathbb{Q}$ be an elliptic curve with a $\mathbb{Q}$-rational isogeny of degree $37$. 
          \begin{enumerate}
              \item If $p \neq 2, 5, 7$ is a prime, then $c_p(E)=1,2,$ or $4$.
               \item The curve $E/\mathbb{Q}$ has $c_2(E)=1, 2,$ or $4$.
              \item The curve $E/\mathbb{Q}$ has $c_5(E)=2$.
              \item The curve $E/\mathbb{Q}$ has $c_7(E)=1$ or $3$.
          \end{enumerate}
          In particular, $c(E)=2^n3^m$, for some $n \geq 2$ and $m \in \{0,1\}$.
      \end{proposition}

\begin{remark}
     The case where $c_2(E)=17$ in Proposition \ref{propdegree17} does indeed occur. Consider the elliptic curve $E/\mathbb{Q}$ with LMFDB label \href{https://www.lmfdb.org/EllipticCurve/Q/130050/gu/1}{130050.gu1}. The curve $E/\mathbb{Q}$ has a $\mathbb{Q}$-rational isogeny of degree $17$ and $c_2(E)=17$.
\end{remark}

We will now prove analogous results for $N=14, 15,  21, 27$. The method that will be used is similar to the one used in \cite[Section 3]{mentzelosisogenies}. Let us briefly explain some of the main ideas here. Let $N$ be as above and let $X_0(N)/\mathbb{Q}$ be the modular curve parametrizing elliptic curves together with an isogeny of degree $N$. In \cite[Table 4]{lozanorobledo}, we can find the $j$-invariants corresponding to non-cuspidal $\mathbb{Q}$-rational points of $X_0(N)/\mathbb{Q}$, i.e., the $j$-invariants of elliptic curves defined over $\mathbb{Q}$ that have a $\mathbb{Q}$-rational isogeny of degree $N$. 
     
     It follows from \cite[Corollary X.5.4.1]{aec}  that all elliptic curves having the same $j$-invariant are twists of each other. Since all these $j$-invariants coming from \cite[Table 4]{lozanorobledo}, are not equal to $0$ or $1728$, we only need to consider quadratic twists. Finally, we can use Table \ref{table} to infer information for Tamagawa numbers of quadratic twists of elliptic curves.

     If $E/\mathbb{Q}$ is an elliptic curve and $d$ is a square-free integer, then we will denote by $E^d/\mathbb{Q}$ the quadratic twist of $E/\mathbb{Q}$ by $d$. We recall now some of the results from \cite{com} for future reference.

     \begin{lemma}\label{results1comalada}(See \cite[Proposition 1]{com}) Let $E/\mathbb{Q}$ be an elliptic curve and $d$ a square-free integer. If $p \neq 2$ is a prime with $p \mid d$, then the reduction types of $E/\mathbb{Q}$ and $E^d/\mathbb{Q}$ modulo $p$ are related as follows 
     \begin{table}[ht]\label{table2}
     \begin{center}
    \begin{tabular}{ ||c|c || } 
    \hline \hline
    Reduction type of $E/\mathbb{Q}$ modulo $p$ & Reduction type of $E^d/\mathbb{Q}$ modulo $p$ \\ 
    \hline
    \hline
     \textup{I}$_0$ & \textup{I}$_0^*$ \\ 
    \textup{I}$_n$ & \textup{I}$_n^*$  \\
    \textup{II} & \textup{IV}$^*$  \\
    \textup{III} & \textup{III}$^*$  \\
    \textup{IV} & \textup{II}$^*$  \\
    \textup{I}$_0^*$  & \textup{I}$_0$  \\
    \textup{II}$^*$  & \textup{IV}  \\
    \textup{III}$^*$  & \textup{III}  \\
    \textup{IV}$^*$  & \textup{II}  \\
    \hline
   \end{tabular}
\caption{Reduction types and Quadratic Twists}
   \label{table2}
\end{center}
\end{table}
\end{lemma}

We will also need the following.

\begin{lemma}\label{lemmareduction}
     Let $E/\mathbb{Q}$ be an elliptic curve and $d$ be a square-free number. 
     \begin{enumerate}
     \item Let $p$ be an odd prime with $p \nmid d$ such that $E/\mathbb{Q}$ has good reduction modulo $p$. Then $E^d/\mathbb{Q}$ has good reduction modulo $p$.
         \item If $E/\mathbb{Q}$ has good reduction modulo $2$, then $E^d/\mathbb{Q}$ has either good reduction or reduction of type \textup{I}$_4^*$, \textup{I}$_8^*$, \textup{II}, or \textup{II}$^*$ modulo $2$.
         \item If $E/\mathbb{Q}$ has modulo $2$ reduction of type \textup{I}$_n$ for some $n > 0$, then $E^d/\mathbb{Q}$ has modulo $2$ either reduction of type \textup{I}$_n$ or reduction of type \textup{I}$_n^*$.
     \end{enumerate}
\end{lemma}

\begin{proof}
    Part $(i)$ essentially follows from Tate's algorithm \cite{tatealgorithm} and is well known to the experts (see \cite[Proposition 1]{com}). Parts $(ii)$ and $(iii)$ can be proved by combining explicit formulas for quadratic twists of elliptic curves (see \cite[Proposition 5.7.1]{ellipticcurvehandbook}) with \cite[Tableau IV]{pap}. Alternatively, one can use \cite[Theorem 4.2]{lorenzini2013} and \cite{haiyang}.
\end{proof}

\begin{proposition}
    Let $E/\mathbb{Q}$ be an elliptic curve with a $\mathbb{Q}$-rational isogeny of degree $14$.
    \begin{enumerate}
        \item If $p \neq 2, 7$ is a prime, then $E/\mathbb{Q}$ has either good reduction or reduction of type \textup{I}$_0^*$ modulo $p$. Moreover, $c_p(E)=1,2,$ or $4$.
        \item The curve $E/\mathbb{Q}$ has reduction of type \textup{III} or \textup{III}$^*$ modulo $7$, and $c_7(E)=2$.
        \item The curve $E/\mathbb{Q}$ has either good reduction or reduction of type \textup{I}$_4^*$, \textup{I}$_8^*$, \textup{II}, or \textup{II}$^*$ modulo $2$.  Moreover, $c_2(E)=1,2,$ or $4$.
    \end{enumerate}

    In particular, $c(E)=2^n$, for some $n \geq 1$.
\end{proposition}

\begin{proof}
    From \cite[Table 4]{lozanorobledo} we find that if $E/\mathbb{Q}$ is an elliptic curve with a $\mathbb{Q}$-rational isogeny of degree $14$, then its $j$-invariant $j(E)$ is equal to either $-3^3\cdot 5^3$ or $3^3 \cdot 5^3 \cdot 17^3$. The curve $E_1/\mathbb{Q}$ with LMFDB label \href{https://www.lmfdb.org/EllipticCurve/Q/49/a/1}{49.a1} is a curve with the smallest conductor in the twist class with $j$-invariant $3^3 \cdot 5^3 \cdot 17^3$. Moreover, the curve $E_2/\mathbb{Q}$ with LMFDB label \href{https://www.lmfdb.org/EllipticCurve/Q/49/a/2}{49.a2} is a curve with the smallest conductor in the twist class with $j$-invariant $-3^3\cdot 5^3$. Since $j(E) \neq 0, 1728$, we know that $E/\mathbb{Q}$ is a quadratic twist of either $E_1/\mathbb{Q}$ or $E_2/\mathbb{Q}$. We note that both $E_1/\mathbb{Q}$ and $E_2/\mathbb{Q}$ have conductor $7^2$ and, hence, they have good reduction away from $7$.

    Proof of $(i)$: Let $p \neq 2, 7$ be a prime. Since both $E_1/\mathbb{Q}$ and $E_2/\mathbb{Q}$ have good reduction modulo $p$, it follows from Table \ref{table2} and Lemma \ref{lemmareduction} that $E/\mathbb{Q}$ has either good reduction or reduction of type \textup{I}$_0^*$ modulo $p$. Thus, from Table \ref{table} we get that $c_p(E)=1,2,$ or $4$.

    Proof of $(ii)$:  The curves $E_1/\mathbb{Q}$ and $E_2/\mathbb{Q}$ both  have reduction of type \textup{III}$^*$. By table \ref{table2}, we know that every quadratic twist of $E_1/\mathbb{Q}$ or $E_2/\mathbb{Q}$, has reduction of type \textup{III} or \textup{III}$^*$ modulo $7$. From Table \ref{table} we see that $c_7(E)=2$.

    Proof of $(iii)$: Since both $E_1/\mathbb{Q}$ and $E_2/\mathbb{Q}$ have good reduction modulo $2$, it follows from Table \ref{table2} and Lemma \ref{lemmareduction} that $E/\mathbb{Q}$ has either good reduction or reduction of type \textup{I}$_4^*$, \textup{I}$_8^*$, \textup{II}, or \textup{II}$^*$ modulo $2$. Finally, from Table \ref{table} we get that $c_p(E)=1,2,$ or $4$.
\end{proof}

\begin{remark}
    The previous proposition is a more precise version of \cite[Proposition 3.7]{trbovic} where it was proved that if $E/\mathbb{Q}$ is an elliptic curve with a $\mathbb{Q}$-rational isogeny of degree $14$, then $2$ divides $c(E)$. 
\end{remark}

\begin{proposition}\label{propdegree15}
    Let $E/\mathbb{Q}$ be an elliptic curve with a $\mathbb{Q}$-rational isogeny of degree $15$.
     \begin{enumerate}
        \item If $p \neq 2, 5$ is a prime, then $E/\mathbb{Q}$ has either good reduction or reduction of type \textup{I}$_0^*$ modulo $p$. Moreover, $c_p(E)=1,2,$ or $4$.
        \item The curve $E/\mathbb{Q}$ has reduction of type \textup{II}, \textup{II}$^*$, \textup{IV} or \textup{IV}$^*$ modulo $5$. Moreover, $c_5(E)=1$ or $3$.
        \item The curve $E/\mathbb{Q}$ has reduction of type \textup{I}$_1$, \textup{I}$^*_1$, \textup{I}$_3$, \textup{I}$^*_3$, \textup{I}$_5$, \textup{I}$^*_5$, \textup{I}$_{15}$, or \textup{I}$^*_{15}$ modulo $2$. Moreover, $c_2(E)=1,3,5$ or $15$.
    \end{enumerate}
    In particular, $c(E)=2^n3^m5^k$, for some $n \geq 1$, $m \in \{0,1,2 \}$, and $k \in \{0,1\}$.
\end{proposition}
\begin{proof}
    From \cite[Table 4]{lozanorobledo} we find that if $E/\mathbb{Q}$ is an elliptic curve with a $\mathbb{Q}$-rational isogeny of degree $15$, then its $j$-invariant $j(E)$ is equal to $\frac{-5^2}{2}$, $\frac{-5^2\cdot 241^3}{2^3}$, $\frac{-5 \cdot 29^3}{2^5}$, or $\frac{5 \cdot 211^3}{2^{15}}$. The curves $E_1/\mathbb{Q}$, $E_2/\mathbb{Q}$, $E_3/\mathbb{Q}$, and $E_4/\mathbb{Q}$ with LMFDB labels \href{https://www.lmfdb.org/EllipticCurve/Q/50a1/}{50.a3}, \href{https://www.lmfdb.org/EllipticCurve/Q/50a2/}{50.a1}, \href{https://www.lmfdb.org/EllipticCurve/Q/50a3/}{50.a2}, and \href{https://www.lmfdb.org/EllipticCurve/Q/50a4/}{50.a4} are curves of the smallest conductor in the twist class with $j$-invariant $\frac{-5^2}{2}$, $\frac{-5^2\cdot 241^3}{2^3}$, $\frac{-5 \cdot 29^3}{2^5}$, or $\frac{5 \cdot 211^3}{2^{15}}$, respectively. Since $j(E) \neq 0, 1728$, we know that $E/\mathbb{Q}$ is a quadratic twist of $E_i/\mathbb{Q}$, for some $i=1,2,3,4$. We note that all curves above have conductor $2 \cdot 5^2$ and, hence, they have good reduction away from $2$ and $5$.

    Proof of $(i)$: Let $p \neq 2, 5$ be a prime. Since $E_i/\mathbb{Q}$ has good reduction modulo $p$, for $i=1,2,3,4$, and $E/\mathbb{Q}$ is a quadratic twist of $E_i/\mathbb{Q}$ for some $i=1,2,3,$ or $4$, it follows from Table \ref{table2} and Lemma \ref{lemmareduction} that $E/\mathbb{Q}$ has either good reduction or reduction of type \textup{I}$_0^*$ modulo $p$. From Table \ref{table} we get that $c_p(E)=1,2,$ or $4$.

    Proof of $(ii)$: The curves $E_1/\mathbb{Q}$, $E_2/\mathbb{Q}$, $E_3/\mathbb{Q}$, and $E_4/\mathbb{Q}$ have reduction of type IV, IV, IV$^*$, and IV$^*$ modulo $5$, respectively. Since $E/\mathbb{Q}$ is a quadratic twist of $E_i/\mathbb{Q}$ for some $i=1,2,3,$ or $4$, from Table \ref{table2} we see that $E/\mathbb{Q}$ has reduction of type \textup{II}, \textup{II}$^*$, \textup{IV} or \textup{IV}$^*$ modulo $5$. Finally, from Table \ref{table} we find that $c_p(E)=1$ or $3$.

    Proof of $(iii)$: The curves $E_1/\mathbb{Q}$, $E_2/\mathbb{Q}$, $E_3/\mathbb{Q}$, and $E_4/\mathbb{Q}$ have reduction of type I$_1$, I$_3$, I$_5$, and I$_{15}$ modulo $2$, respectively. Since $E/\mathbb{Q}$ is a quadratic twist of $E_i/\mathbb{Q}$ for some $i=1,2,3,$ or $4$, from Lemma \ref{lemmareduction} we see that $E/\mathbb{Q}$ has reduction of type \textup{I}$_1$, \textup{I}$^*_1$, \textup{I}$_3$, \textup{I}$^*_3$, \textup{I}$_5$, \textup{I}$^*_5$, \textup{I}$_{15}$, or \textup{I}$^*_{15}$ modulo $2$. The part about $c_2(E)$ follows from Table \ref{table}.
\end{proof}

\begin{remark}
    The case where $c_2(E)=15$ in Proposition \ref{propdegree15} does indeed occur. Consider the elliptic curve $E/\mathbb{Q}$ with LMFDB label \href{https://www.lmfdb.org/EllipticCurve/Q/50/b/4}{50.b4}. The curve $E/\mathbb{Q}$ has a $\mathbb{Q}$-rational isogeny of degree $15$ and $c_2(E)=15$.
\end{remark}

\begin{proposition}\label{propdegree21}
    Let $E/\mathbb{Q}$ be an elliptic curve with a $\mathbb{Q}$-rational isogeny of degree $21$.
    \begin{enumerate}
        \item If $p \neq 2, 3$ is a prime, then $E/\mathbb{Q}$ has either good reduction or reduction of type \textup{I}$_0^*$ modulo $p$. Moreover, $c_p(E)=1,2,$ or $4$.
        \item The curve $E/\mathbb{Q}$ has reduction of type \textup{I}$_1$, \textup{I}$_1^*$, \textup{I}$_3$, \textup{I}$_3^*$, \textup{I}$_7$, \textup{I}$_7^*$, \textup{I}$_{21}$, or \textup{I}$_{21}^*$ modulo $2$. Moreover, $c_2(E)=1, 2, 3, 4, 7,$ or $21$.
        \item The curve $E/\mathbb{Q}$ has reduction of type \textup{II}, \textup{II}$^*$, \textup{IV} or \textup{IV}$^*$ modulo $3$. Moreover, $c_3(E)=1$ or $3$.
    \end{enumerate}
     In particular, $c(E)=2^n3^m7^k$, for some $n \geq 1$, $m \in \{0,1,2 \}$, and $k \in \{0,1\}$.
\end{proposition}
\begin{proof}
    From \cite[Table 4]{lozanorobledo} we find that if $E/\mathbb{Q}$ is an elliptic curve with a $\mathbb{Q}$-rational isogeny of degree $21$, then its $j$-invariant $j(E)$ is equal to $-\frac{3^25^6}{2^3}$, $-\frac{3^35^3}{2}$, $-\frac{3^25^3101^3}{2^{21}}$, or $-\frac{3^3 5^3 383^3}{2^7}$. The curves $E_1/\mathbb{Q}$, $E_2/\mathbb{Q}$, $E_3/\mathbb{Q}$, and $E_4/\mathbb{Q}$ with LMFDB labels \href{https://www.lmfdb.org/EllipticCurve/Q/162b1/}{162.c3}, \href{https://www.lmfdb.org/EllipticCurve/Q/162b2/}{162.c4}, \href{https://www.lmfdb.org/EllipticCurve/Q/162b3/}{162.c2}, and \href{https://www.lmfdb.org/EllipticCurve/Q/162b4/}{162.c1} are curves of the smallest conductor in the twist class with $j$-invariant $-\frac{3^25^6}{2^3}$, $-\frac{3^35^3}{2}$, $-\frac{3^25^3101^3}{2^21}$, or $-\frac{3^3 5^3 383^3}{2^7}$, respectively.  Since $j(E) \neq 0, 1728$, we know that $E/\mathbb{Q}$ is a quadratic twist of $E_i/\mathbb{Q}$, for some $i=1,2,3,4$. We note that all curves above have conductor $2 \cdot 3^4$ and, hence, they have good reduction away from $2$ and $3$.

    Proof of $(i)$: Let $p \neq 2,3$ be a prime. Since $E_i$ has good reduction modulo $p$ for all $i=1,2,3,4$, then we see from Table \ref{table2} and Lemma \ref{lemmareduction} that $E/\mathbb{Q}$ has either good reduction or reduction of type \textup{I}$_0^*$ modulo $p$. Thus, from Table \ref{table} we get that $c_p(E)=1,2,$ or $4$.

    Proof of $(ii)$: The curves $E_1/\mathbb{Q}$, $E_2/\mathbb{Q}$, $E_3/\mathbb{Q}$, and $E_4/\mathbb{Q}$ have reduction of type \textup{I}$_3$, \textup{I}$_1$, \textup{I}$_{21}$, and \textup{I}$_7$ modulo $2$, respectively. Since $E/\mathbb{Q}$ is a quadratic twist of $E_i/\mathbb{Q}$ for some $i=1,2,3,$ or $4$, from Table \ref{table2} we see that $E/\mathbb{Q}$ has reduction of type \textup{I}$_1$, \textup{I}$_1^*$, \textup{I}$_3$, \textup{I}$_3^*$, \textup{I}$_7$, \textup{I}$_7^*$, \textup{I}$_{21}$, or \textup{I}$_{21}^*$ modulo $2$. Finally, from Table \ref{table} we find that $c_2(E)=1,2, 3, 4, 7,$ or $21$.

    Proof of $(iii)$: The curves $E_1/\mathbb{Q}$, $E_2/\mathbb{Q}$, $E_3/\mathbb{Q}$, and $E_4/\mathbb{Q}$ have reduction of type \textup{II}, \textup{II}$^*$, \textup{II}, and \textup{II}$^*$ modulo $3$, respectively. Since $E/\mathbb{Q}$ is a quadratic twist of $E_i/\mathbb{Q}$ for some $i=1,2,3,$ or $4$, from Table \ref{table2} we see that $E/\mathbb{Q}$ has reduction of type \textup{II}, \textup{II}$^*$, \textup{IV}, \textup{IV}$^*$ modulo $3$. Finally, we have that $c_3(E)=1$ or $3$ from Table \ref{table}.
\end{proof}

\begin{remark}
    The case where $c_2(E)=21$ in Proposition \ref{propdegree21} does indeed occur. Consider the elliptic curve $E/\mathbb{Q}$ with LMFDB label \href{https://www.lmfdb.org/EllipticCurve/Q/162/c/2}{162.c2}. The curve $E/\mathbb{Q}$ has a $\mathbb{Q}$-rational isogeny of degree $21$ and $c_2(E)=21$.
\end{remark}

\begin{proposition}
    Let $E/\mathbb{Q}$ be an elliptic curve with a cyclic $\mathbb{Q}$-rational isogeny of degree $27$.
     \begin{enumerate}
        \item If $p \neq 2, 3$ is a prime, then $E/\mathbb{Q}$ has either good reduction or reduction of type \textup{I}$_0^*$ modulo $p$. Moreover, $c_p(E)=1,2,$ or $4$.

        \item The curve $E/\mathbb{Q}$ has either good reduction or reduction of type \textup{I}$_4^*$, \textup{I}$_8^*$, \textup{II}, or \textup{II}$^*$ modulo $2$.  Moreover, $c_2(E)=1,2,$ or $4$.

        \item The curve $E/\mathbb{Q}$ has reduction of type \textup{II}, \textup{II}$^*$, \textup{IV}, \textup{IV}$^*$ modulo $3$. Moreover, $c_3(E)=1$ or $3$.

    \end{enumerate}

In particular, $c(E)=2^n3^m$, for some $n \geq 0$ and $m \in \{0,1\}$.
\end{proposition}
\begin{proof}
    From \cite[Table 4]{lozanorobledo} we find that if $E/\mathbb{Q}$ is an elliptic curve with a $\mathbb{Q}$-rational isogeny of degree $14$, then its $j$-invariant $j(E)$ is equal to $-2^{15} \cdot 3 \cdot 5^3$. The curve $E_1/\mathbb{Q}$ with LMFDB label \href{https://www.lmfdb.org/EllipticCurve/Q/27/a/1}{27.a1} is a curve with the smallest conductor in the twist class with $j$-invariant $-2^{15} \cdot 3 \cdot 5^3$. Since $j(E) \neq 0, 1728$, we know that $E/\mathbb{Q}$ is a quadratic twist of $E_1/\mathbb{Q}$. We note that $E_1/\mathbb{Q}$ has conductor $3^3$ and, hence, it has good reduction away from $3$.

    Proof of $(i)$: Let $p \neq 2, 3$ be a prime. Since $E_1/\mathbb{Q}$ has good reduction modulo $p$, it follows from Table \ref{table2} and Lemma \ref{lemmareduction} that $E/\mathbb{Q}$ has either good reduction or reduction of type \textup{I}$_0^*$ modulo $p$. Thus, from Table \ref{table} we get that $c_p(E)=1,2,$ or $4$.

    Proof of $(ii)$: Since $E_1/\mathbb{Q}$ has good reduction modulo $2$, it follows from Table \ref{table2} and Lemma \ref{lemmareduction} that $E/\mathbb{Q}$ has either good reduction or reduction of type \textup{I}$_4^*$, \textup{I}$_8^*$, \textup{II}, or \textup{II}$^*$ modulo $2$. Finally, from Table \ref{table} we get that $c_p(E)=1,2,$ or $4$.

    Proof of $(iii)$: The curve $E_1/\mathbb{Q}$ has reduction of type \textup{II}$^*$ modulo $3$. Moreover, since $E/\mathbb{Q}$ is a quadratic twist of $E_1/\mathbb{Q}$, by Table \ref{table2}, we find that $E/\mathbb{Q}$ has reduction of type \textup{II}$^*$ or \textup{IV} modulo $3$. Finally, from Table \ref{table} we see that $c_3(E)=1$ or $3$.
\end{proof}

\begin{remark}\label{rmktamagawa}
    We now explain how one can use quadratic twists to search examples of elliptic curves with isogenies that exhibit higher $2$-divisibility in their Tamagawa numbers. Let $E/\mathbb{Q}$ be an elliptic curve with a $\mathbb{Q}$-rational isogeny of degree $\ell$. For simplicity assume that $E/\mathbb{Q}$ has good reduction modulo $2,3$ and that $E/\mathbb{Q}$ does not multiplicative reduction modulo any prime. Let $d>3$ be an integer  not divisible by $ 3$ with $d \equiv 1 (\text{mod} \: 4 )$ and such that $E/\mathbb{Q}$ has good reduction modulo every prime that divides $d$. Consider now the quadratic twist $E^d/\mathbb{Q}$ of $E/\mathbb{Q}$. We know that $E^d/\mathbb{Q}$ will also have a $\mathbb{Q}$-rational isogeny of degree $\ell$.

    Denote by $\Delta_E$ the minimal discriminant of $E/\mathbb{Q}$ and by $\Delta_{E^d}$ the minimal discriminant of $E^d/\mathbb{Q}$. Note that $2 \nmid \Delta_E$ because $E/\mathbb{Q}$ has good reduction modulo $2$. In addition, our assumptions on $E/\mathbb{Q}$ and $d$ imply that for every prime $p>3$ with $p \nmid d$ we have that $c_p(E^d)=c_p(E)=1$, see e.g. \cite[Theorem 4.1]{brsttw}. Therefore, we have $$c(E)=\displaystyle\prod_{p\mid \Delta_E} c_p(E)=\displaystyle\prod_{p\mid \Delta_E} c_p(E^d)=\frac{c_2(E^d)\displaystyle\prod_{q \mid d} c_q(E^d)}{c_2(E^d)\displaystyle\prod_{q \mid d} c_q(E^d)}\displaystyle\prod_{p\mid \Delta_E } c_p(E^d)=\frac{1}{c_2(E^d)\displaystyle\prod_{q \mid d} c_q(E^d)} c(E^d).$$ Thus $$c(E^d)=c_2(E^d)\displaystyle\prod_{q \mid d} c_q(E^d)c(E).$$ 
    
    By \cite[Theorem 5.1]{brsttw} we have that $c_2(E^d)=1,2,$ or $4$. Let $q$ be a divisor of $q$. Since $E/\mathbb{Q}$ has good reduction modulo $q$ we have that $c_d(E^d)=1,2,$ or $4$, where the exact value can be explicitly computed and depends on the coefficients of a minimal Weierstrass equation of $E/\mathbb{Q}$. Therefore, by picking $d$ appropriately we can produce examples of elliptic curves with isogenies where a high power of $2$ divides their Tamagawa numbers.
\end{remark}

\begin{proposition}\label{prop5713}
    For each $\ell=5,7,$ or $13$ there exist infinitely many elliptic curves $E/\mathbb{Q}$ with distinct $j$-invariants that have a $\mathbb{Q}$-rational isogeny of degree $\ell$ such that $c(E)=2^n3^m$, for some $n,m$.
\end{proposition}
\begin{proof}
    Let $\ell$ be as in the statement of the proposition. Put $$F_\ell(t)= \begin{cases} 
      \frac{(t^2 + 10t + 5)^3}{t}, & \text{ if } \ell=5, \\
      \frac{(t^2 +13t+49)(t^2 +5t+1)^3}{t}, & \text{ if } \ell=7, \\
      \frac{(t^2 +5t+13)(t^4 +7t^3 +20t^2 +19t+1)^3}{t} & \text{ if } \ell=13. 
   \end{cases}$$
  It follows from \cite[Table 3]{lozanorobledo} (see also \cite[Chapter 4]{Tsukazakithesis}) that for every $t_0 \in \mathbb{Q}$ and every elliptic curve $E/\mathbb{Q}$ with $j$-invariant $j(E)=F_{\ell}(t_0)$ we have that $E/\mathbb{Q}$ has a $\mathbb{Q}$-rational isogeny of degree $\ell$.

  Let now $$a_{4,\ell}(t)= \begin{cases} 
      -3(t^2 + 10t + 5)(t^2 + 22t + 125), & \text{ if } \ell=5, \\
      -3(t^2 +5t+1)(t^2 +13t+49), & \text{ if } \ell=7, \\
      -3(t^2 +5t+13)(t^2 +6t+13)(t^4 +7t^3 +20t^2 +19t+1), & \text{ if } \ell=13. 
   \end{cases}$$

   and 

   $$a_{6,\ell}(t)= \begin{cases} 
      -2(t^2 + 22t + 125)^2(t2 + 4t - 1), & \text{if } \ell=5, \\
      -2(t^2 +13t+49)(t^4 +14t^3 +63t^2 +70t-7), & \text{if } \ell=7, \\
      -2(t^2 +5t+13)(t^2 +6t+13)^2(t^6 +10t^5 +46t^4 +108t^3 +122t^2 + 38t - 1), & \text{if } \ell=13. 
   \end{cases}$$ For each $t_0 \in \mathbb{Z}$ consider the elliptic curve $E_{t_0,\ell}/\mathbb{Q}$ given by the Weierstrass equation $$y^2=x^3+a_{4,\ell}(t)x+a_{6,\ell}(t).$$ It is not hard to see that $E_{t_0,\ell}/\mathbb{Q}$ has $j$-invariant $j(E_{t_0,\ell})=F_{\ell}(t_0)$. We refer the reader to \cite[Section 4.2]{Tsukazakithesis} for more information on the elliptic curves $E_{t_0,\ell}/\mathbb{Q}$.
   
  For every prime $p > \ell$ consider the elliptic curve $E_{p,\ell}/\mathbb{Q}$. For any integer $k$ denote by $v_p(k)$ the $p$-adic valuation of $k$. Since we assume that $p > \ell$, we find that $v_p(j({E_{t_0,\ell}}))=-1.$ Therefore, we see $c_p(E) \leq 4$ by \cite[Corollary 9.2]{silverman2}. Moreover, we have that $v_q(j({E_{t_0,\ell}}))>0 $, for every prime $q \neq p$. Thus, we obtain that $c_q(E) \leq 4$ again by \cite[Corollary 9.2]{silverman2}. This completes our proof.

  \end{proof}

  \begin{remark}
    Let $\ell$ be a prime and let $E/\mathbb{Q}$ be an elliptic curve with a $\mathbb{Q}$-rational isogeny of degree $\ell$ such that $v_p(c(E))=0$ for some prime $p>5$, where $v_p(c(E))$ is the $p$-adic valuation of $c(E)$. Then by considering quadratic twists at primes of good reduction of $E/\mathbb{Q}$ we can find infinitely many elliptic curves $E' /\mathbb{Q}$ a $\mathbb{Q}$-rational isogeny of degree $\ell$ such that $v_p(c(E'))=0$.
\end{remark}

We end this section with a proposition that connects the Tamagawa number with the parity of the rank of an elliptic curve, when the curve has a torsion point of order $7$.

  \begin{proposition}
      Let $E/\mathbb{Q}$ be an elliptic curve and a $\mathbb{Q}$-rational point $P$ of order $7$. Write $E'=E/\left<P\right>$ for the isogenous curve. Assume that the Tate-Shafarevich group $\Sha(E/\mathbb{Q})$ is finite and that $E/\mathbb{Q}$ has semistable reduction modulo $7$. Then  $$v_7(c(E)) \equiv v_7(c(E')) + \mathrm{rk}\: E/\mathbb{Q} + 1 \: (\text{mod }2),$$ where $\mathrm{rk}\: E/\mathbb{Q}$ is the rank of $E/\mathbb{Q}$ and $v_7(c(E))$ and $v_7(c(E'))$ are the $7$-adic valuations of $c(E)$ and $c(E')$, respectively.

  \end{proposition}

  \begin{proof}
      Let $\pi : E \rightarrow E'$ be the associated isogeny. Since we assume that $\Sha(E/\mathbb{Q})$ is finite, by the isogeny invariance of the BSD formula, which is a theorem due to Cassels \cite{Casselsgenus1BSD}, we have $$\frac{\Omega(E)|\Sha(E/\mathbb{Q})| R(E/\mathbb{Q})c(E)}{|E(\mathbb{Q})_{\textrm{tors}}|^2}=\frac{\Omega(E')|\Sha(E'/\mathbb{Q})|R(E'/\mathbb{Q}) c(E')}{|E'(\mathbb{Q})_{\textrm{tors}}|^2}.$$ Here $R(E/\mathbb{Q})$ and $R(E'/\mathbb{Q})$ are the regulators of $E/\mathbb{Q}$ and $E'/\mathbb{Q}$, respectively. Moreover, $\Omega(E):=\displaystyle\int_{E(\mathbb{R})}|\omega_{min}|$ and $\Omega(E'):=\displaystyle\int_{E'(\mathbb{R})}|{\omega'}_{min}|$, where $\omega_{min}$ and $\widehat{\omega}_{min}$ are two minimal invariant differentials on $E/\mathbb{Q}$ and $\widehat{E}/\mathbb{Q}$ respectively (see Section III.1 of \cite{aec} and page 451 of \cite{aec}).

      By \cite[Lemma 1.3]{notesonparity} we find that $$v_7\left(\frac{R(E/\mathbb{Q})}{R(E'/\mathbb{Q})}\right)=\mathrm{rk}\: E/\mathbb{Q}+2n,$$ for some $n\in \mathbb{Z}$. Moreover, by \cite[Lemma 8.5]{dd} we obtain that $$v_7\left(\frac{\Omega(E)}{\Omega(E')}\right)=1.$$ Since the modular curve $X_0(49)$ has no non-cuspidal rational points, proceeding the same way as in \cite[Proof of Theorem 2.1]{dummigan}, we find that $E'/\mathbb{Q}$ cannot have a $\mathbb{Q}$-rational point of order $7$. consequently, $v_7(|E'(\mathbb{Q})_{\textrm{tors}}|)=0$. Finally, we know from \cite[Corollary 17.2.1]{aec} that $v_7(\Sha(E/\mathbb{Q}))$ and $v_7(\Sha(E/\mathbb{Q}))$ are even. Putting everything together, we arrive at the desired equality.

  \end{proof}


\section{Small Tamagawa numbers and Torsion}

If $E/\mathbb{Q}(T)$ be an elliptic curve, then for almost all values $t \in \mathbb{Q}$ the specialization $E_t/\mathbb{Q}$ of $\mathcal{E}$ at $T=t$ is an elliptic curve. Therefore, one can think of $E/\mathbb{Q}(T)$ as an 1-parameter family of elliptic curves. In this chapter we will produce infinite families of elliptic curves with small Tamagawa numbers and torsion points, by using appropriate specializations. The key ingredient here is a result from analytic number theory on almost prime values of polynomials, combined with explicit equations of elliptic curves with torsion points.

Throughout this section, for any integer $k$, we will denote by $v_p(k)$ the $p$-adic valuation of $k$. We will also denote by $P_r$ the set of positive integers with at most $r$ prime divisors, counted with multiplicity. In other words, $n \in P_r$ if and only if $\Omega(n) \leq r$, where $\Omega(n)$ is the total number of prime factors of $n$ (counted with multiplicity). If $F(x)$ is a polynomial with integer coefficients, then we will write $\rho_F(p)$ for the number of solutions of the congruence $$F(x) \equiv 0 \: ( \text{mod } p).$$ The following theorem is a corollary of \cite[Theorem 9.8]{sievemethodsbook}.
 
\begin{theorem}\label{thmdivisors}
    Let $F(x) \neq \pm x$ be an irreducible polynomial of degree $g \geq  1$  with integral coefficients such that $\rho_F(p)<p$ for every prime $p$. Assume also that $\rho_{F}(p)<p-1$ for every prime $p \leq \deg (F)+1$ with $p \nmid F(0)$. Then there exist infinitely many prime numbers $p$ such that $F(p)$ has at most $2g+1$ prime factors.

\end{theorem}

Using the previous theorem we can prove the following result.

\begin{proposition}\label{propunconditional}
    There exist infinitely many elliptic curves $E/\mathbb{Q}$ with a $\mathbb{Q}$-rational point of order $4$ and such that $c(E) =4, 8,$ or $12$.
\end{proposition}
\begin{proof}
    Consider the elliptic curve $E/\mathbb{Q}(T)$ given by the following Weierstrass equation $$E \: : \: y^2+xy- T y=x^3- T x^2.$$ The curve $E/\mathbb{Q}(T)$ has discriminant $$\Delta(T)=T^4(1+16T), $$ and $c_4$-invariant $$c_4(T)=16T^2+16T+1.$$ 
    
    For each $\lambda \in \mathbb{Z}$ such that $\Delta(\lambda)=\lambda^4(1+16\lambda) \neq 0$ the specialization $E_\lambda$ of $E/\mathbb{Q}(T)$ at $T=\lambda$ is an elliptic curve over $\mathbb{Q}$. Moreover, each elliptic curve $E_\lambda/\mathbb{Q}$ has a $\mathbb{Q}$-rational point of order $4$ (see \cite[Section 4.4]{hus}).

    It is immediate to check that the polynomial $16x+1$ satisfies the hypothesis of Theorem \ref{thmdivisors}. Therefore, by Theorem \ref{thmdivisors} we obtain that there exist infinitely many primes $p$ such that $1+16p \in P_3$. Let $p>2$ be a prime such that $1+16p \in P_3$ and consider the specialization $E_p/\mathbb{Q}$ of $E/\mathbb{Q}(T)$ at $T=p$. From the Weierstrass equation of $E_p/\mathbb{Q}$ we can see that it has split multiplicative reduction modulo $p$ with $c_p(E_p)=v_p(\Delta(p))=4$. 
    
    Let now $q$ be any prime that divides $\Delta(p)=p^4(1+16p)$. Since $\mathrm{gcd}(c_4(p), \Delta(p))=1$, we find that $E_p/\mathbb{Q}$ has multiplicative reduction modulo $q$. Since $1+16p \in P_3$, we know that $1+16p=q_1q_2q_3$, where $q_1$, $q_2$, and $q_3$ are (not necessarily distinct) primes. If $q_1=q_2=q_3$, then $c_{q_1}(E_p)=1$ or $3$. If two of the are equal, say $q_1$ and $q_3$, then $c_{q_1}(E_p)c_{q_2}(E_p)=1$ or $2$, and if all the $q_i$ are pairwise distinct, then $c_{q_1}(E_p)c_{q_2}(E_p)c_{q_3}(E_p)=1$. Finally, since $E_p/\mathbb{Q}$ has good reduction away from the primes that divide $\Delta(p)$, we see that $c(E_p)=4, 8,$ or $12$.
\end{proof}

\begin{remark}
    As is noted in \cite[Remark 2.5]{lor}, conditionally on Schinzel's Hypothesis H, there are infinitely many non-isomorphic elliptic curves $E/\mathbb{Q}$ with a $\mathbb{Q}$-rational point of order $4$ and such that $c_{E} =2$. The significance of Proposition \ref{propunconditional} is that it is an unconditional result. 
\end{remark}

\begin{proposition}\label{propfivetorsion}
    There exist infinitely many elliptic curves $E/\mathbb{Q}$ with a $\mathbb{Q}$-rational point of order $5$ such that $c(E) \leq 30$.
\end{proposition}
\begin{proof}
    Consider the elliptic curve $E/\mathbb{Q}(T)$ given by the following Weierstrass equation $$E \: : \: y^2+(1-T)xy- T y=x^3- T x^2.$$ The curve $E/\mathbb{Q}(T)$ has discriminant $$\Delta(T)=T^5(T^2-11T-1), $$ and $c_4$-invariant $$c_4(T)=T^4-12T^3+14T^2+12T+1.$$ The resultant of $\Delta(T)$ and $c_4(T)$ (see also \cite[Page 2001]{lor}) is $$\mathrm{Res}(\Delta(T),c_4(T))=5^2.$$

    For each $\lambda \in \mathbb{Z}$ such that $\Delta(\lambda)\neq 0$ the specialization $E_\lambda$ of $E/\mathbb{Q}(T)$ at $T=\lambda$ is an elliptic curve over $\mathbb{Q}$. Moreover, each elliptic curve $E_\lambda/\mathbb{Q}$ has a $\mathbb{Q}$-rational point of order $5$ (see \cite[Section 4.4]{hus}).

     Since the polynomial $x^2-11x-1$ satisfies the hypothesis of Theorem \ref{thmdivisors}, we find that there exist infinitely many primes $p$ such that $p^2-11p-1 \in P_5$. Let now $p > 5 $ such that $p^2-11p-1 \in P_5$ and consider the specialization $E_p/\mathbb{Q}$ of $E/\mathbb{Q}(T)$ at $T=p$. Since $\mathrm{Res}(\Delta(T),c_4(T))=5^2$, we have that if a prime divides both $\Delta(p)$ and $c_4(p)$, then it must be equal to $5$. Moreover, since $p > 5$ and $p^2-11p-1 \not\equiv 0 \: (\text{mod }5) $, we see that $E_p/\mathbb{Q}$ has good reduction modulo $5$. Therefore, $E_p/\mathbb{Q}$ has everywhere semi-stable reduction.

      From the Weierstrass equation of $E_p/\mathbb{Q}$ we can see that it has split multiplicative reduction with $c_p(E_p)=v_p(\Delta(p))=5$. Moreover, since $p^2-11p-1 \in P_5$, we know that $p^2-11p-1=q_1q_2q_3q_4q_5$, where $q_1,q_2,q_3,q_4,q_5$ are (not necessarily distinct) primes. Since $c_{q_i}(E_p) \leq v_{q_i}(\Delta(p))$ for every $i$, we obtain that $\displaystyle\prod_{\substack{ q \text{ prime} \\ q \: \mid \: p^2-11p-1}}c_q(E_p) \leq 6$. Finally, since $E_p/\mathbb{Q}$ has good reduction away from the primes that divide $\Delta(p)$, we see that $c(E) \leq 30$. This proves our proposition.
\end{proof}

\begin{proposition}\label{propositiondivisors}
     Let $N'=4,5,6,7,8,9,10$ or $12$. For every prime $p$ there exist infinitely many elliptic curves $E/\mathbb{Q}$ with a cyclic $\mathbb{Q}$-rational isogeny of order $N'$ such that $p$ divides $c(E)$.
\end{proposition}
\begin{proof}
    Let $N'$ be as in the proposition and $p$ be a prime. We will show a slightly statement. We will show that there exist infinitely many elliptic curves $E/\mathbb{Q}$ with a $\mathbb{Q}$-rational point of order $N'$ such that $p$ divides $c(E)$.

    For any elliptic curve with a $\mathbb{Q}$-rational point $P$ of order $N'$, there exists a parameter $\lambda \in \mathbb{Q}$ such that it is of the form
    $$E_{N', \lambda} \::\: y^2+(1-c_{N'}(\lambda))xy-b_{N'}(\lambda)y=x^3-b_{N'}(\lambda)x^2,$$ for some explicit functions $b_{N'}(\lambda),c_{N'}(\lambda)$ of $\lambda$, see \cite[Section 2]{lor} for the precise formulas. Conversely, for each choice of $\lambda \in \mathbb{Q}$ the curve $E_{N', \lambda}/\mathbb{Q}$ has a $\mathbb{Q}$-rational point of order $N'$. The reader is referred to \cite[Section 4.4]{hus} and \cite[Section 2]{lor} for more information.

    \begin{lemma}\label{lemmasplit}
        If $q$ is a prime such that $v_q(\lambda)=p$, then $E_{N', \lambda}/\mathbb{Q}$ has split multiplicative reduction modulo $q$ with $p \mid c_q(E_{N', \lambda})$.
    \end{lemma}
    \begin{proof}[Proof of the lemma]
        By looking at the explicit equations for the curve $E_{N', \lambda}/\mathbb{Q}$, we see that if $v_q(\lambda)>0$, then $v_q(b_{N'}(\lambda)),v_q(c_{N'}(\lambda))>0$, see \cite[Section 2]{lor}. Therefore, $E_{N', \lambda}/\mathbb{Q}$ has split multiplicative reduction modulo $q$. Moreover, $c_q(E_{N', \lambda})=v_q(\Delta_{E_{N', \lambda}})=N'p$. Therefore,  we have that $p$ divides $c_q(E)$.
    \end{proof}
    Let now $p$ be a prime. For every prime $q$, pick $\lambda=q^p$. Then, according to Lemma \ref{lemmasplit} the curve $E_{N', q^p}/\mathbb{Q}$ has split multiplicative reduction modulo $q$ with $p \mid c_q(E_{N', \lambda})$. This completes the proof of our proposition.
\end{proof}

\section{A more general consideration}

In this section, using a result on almost prime values of polynomials, we prove Theorem \ref{generalthm}. We then present some improvements of the bound given by Theorem \ref{generalthm} when the discriminant is of special form. 

Throughout this section, for any integer $k$ we will denote by $v_p(k)$ the $p$-adic valuation of $k$. Recall also that if $F(x)$ is a polynomial with integer coefficients, then we write $\rho_F(p)$ for the number of solutions of the congruence $$F(x) \equiv 0 \: ( \text{mod } p).$$ In what follows we will denote by $P_s$ the set of all natural numbers that have at most $s$ prime factors. Thus for $n\in \mathbb{N}$ we have that $n \in P_s$ if and only if $\Omega(n) \leq s$, where $\Omega(n)$ is the total number of prime factors of $n$ (counted with multiplicity).
 
 The following theorem is an immediate corollary of \cite[Theorem 10.4]{sievemethodsbook}.
\begin{theorem}\label{sievesbookthm}
   Let $F_1(x), F_2(x),...,F_g(x) $ be distinct irreducible polynomials with integral coefficients and write $F(x)=F_1(x)F_2(x) \cdots F_g(x)$ for their product. Assume that $\rho_F(p)<p$ for every prime $p$. Then there exists a positive integer $s$ that can be explicitly computed and depends on $F$ such that there exist infinitely many $n \in \mathbb{N}$ with $F(n) \in P_s$.
\end{theorem}

\begin{proof}[Proof of Theorem \ref{generalthm}]
    Let $E/\mathbb{Q}(T)$ be a non-isotrivial elliptic curve and let $\Delta(T)$ be the discriminant of a Weierstrass equation for $E/\mathbb{Q}(T)$ which is minimal at all places of $\mathbb{Q}[T]$. According to Theorem \ref{sievesbookthm} there exists a positive integer $s$ that can be explicitly computed and depends on $F$ such that there exist infinitely many $n \in \mathbb{N}$ with $F(n) \in P_s$. 
    
    Let now $n \in \mathbb{N}$ such that $F(n) \in P_s$ and consider the specialization $E_n/\mathbb{Q}$ of $E/\mathbb{Q}(T)$ at $T=n$. For every prime $p \geq 5$ that divides $\Delta(n)$ we have that $c_p(E_n) \leq v_p(\Delta(n))$, see \cite[Page 365]{silverman2}. Assume now that $p=2$ or $3$, and that $p$ divides $\Delta(n)$. If $E_n/\mathbb{Q}$ has multiplicative reduction modulo $p$, then $c_p(E_n) \leq v_p(\Delta(n))$, by \cite[Page 366]{silverman2}. On the other hand, if $E_n/\mathbb{Q}$ has additive reduction modulo $p$, then $c_p(E_n) \leq 4 \leq 4 v_p(\Delta(n))$. 

Moreover, since $F(n) \in P_s$, we have that $$v_p(\Delta(n))=v_p(m F_1(n)^{m_1}F_2(n)^{m_2} \cdots F_g(n)^{m_q}) \leq v_p(m) + s  \mathrm{deg}(\Delta) \leq \log_2(m)+s\deg(\Delta),$$ for every prime $p$. Here in the last step we used the elementary inequality $v_p(m) \leq \log_2 (m)$. Combining this with the observations of the previous paragraph we obtain that \begin{align*}
    c(E_n)= \displaystyle\prod_{p \mid \Delta(n)} c_p(E_n)  \leq 16 \displaystyle\prod_{p \mid \Delta(n)} v_p(\Delta(n)) \leq 16 \displaystyle\prod_{p \mid \Delta(n)} (\log_2(m)+s\deg(\Delta))  \\
    \leq  16 (\log_2(m)+ s \: \mathrm{deg}(\Delta))^{d(m)+s}
\end{align*}  for each $n \in \mathbb{Z}$ such that $F(n) \in P_s$, where $d(m)$ is the number of prime divisors of $m$. This proves our theorem.
    
\end{proof}

The following result is a sharper version of Theorem \ref{generalthm} when the discriminant is of special form.

\begin{theorem}\label{theoremspecialization}
    Let $E/\mathbb{Q}(T)$ be a non-isotrivial elliptic curve and let $\Delta(T)$ be the discriminant of a Weierstrass equation for $E/\mathbb{Q}(T)$ which is minimal at all places of $\mathbb{Q}[T]$. 
    \begin{enumerate}
        \item Assume $\Delta(T)= m g(T)$, where $m \in \mathbb{Z}$ and $g(T)$ is an irreducible quadratic polynomial whose coefficients are integers with highest common factor $1$. Then there exist infinitely many $n \in \mathbb{N}$ such that $c(E_n) \leq 16 (\log_2(m)+1)^{d(m)}$, where $d(m)$ is the number of prime divisors of $m$.
        \item Assume $\Delta(T)=mg(T)$, where $m \in \mathbb{Z}$ and $g(t)$ is a cubic polynomial whose coefficients are integers with highest common factor $1$ and is not divisible by the 2nd power of a linear polynomial with integral coefficients. Then there exist infinitely many $n \in \mathbb{N}$ such that $c(E_n) \leq 16 (\log_2(m)+1)^{d(m)}$, where $d(m)$ is the number of prime divisors of $m$.
    \end{enumerate}
\end{theorem}
\begin{proof}
    Let $E/\mathbb{Q}(T)$ be a non-isotrivial elliptic curve and let $\Delta(T)$ be the discriminant of a Weierstrass equation for $E/\mathbb{Q}(T)$ which is minimal at all places of $\mathbb{Q}[T]$.   
    
    { \it Proof of $(i)$:} Assume that $\Delta(T)=mg(T)$, where $g(T)$ is as in $(i)$. According to \cite{ricci1933} (see also \cite[Pages 416 and 417]{erdos1953}) we have that there exist infinitely many positive integers $n$ for which for $g(n)$ is square-free. Let $n \in \mathbb{N}$ such that $g(n)$ is square-free and consider the specialization $E_n/\mathbb{Q}$ of $E/\mathbb{Q}(T)$ at $T=n$. 

    We have $$c(E_n)=\displaystyle\prod_{p \mid mg(n)} c_p(E_n) = \displaystyle\prod_{ \substack{ p\mid m \\ p \: \nmid g(n)}}c_p(E_n) \displaystyle\prod_{ \substack{ p\: \nmid m \\ p \mid g(n)}}c_p(E_n) \displaystyle\prod_{ \substack{ p\mid m \\ p \mid g(n)}}c_p(E_n). $$

    Because $g(n)$ is square free we have that $c_p(E_n)=v_p(\Delta(n))=1$ for every prime $p$ with $p \mid g(n)$ and $p \nmid m$. Moreover, for every prime $p \geq 5$ that divides $\Delta(n)$ we have that $c_p(E_n) \leq v_p(\Delta(n))$, see \cite[Page 365]{silverman2}. Assume now that $p=2$ or $3$, and that $p$ divides $\Delta(n)$. If $E_n/\mathbb{Q}$ has multiplicative reduction modulo $p$, then $c_p(E_n) \leq v_p(\Delta(n))$, by \cite[Page 366]{silverman2}. On the other hand, if $E_n/\mathbb{Q}$ has additive reduction modulo $p$, then $c_p(E_n) \leq 4 \leq 4 v_p(\Delta(n))$. 

    Therefore,

$$\displaystyle\prod_{ \substack{ p\mid m \\ p \: \nmid g(n)}}c_p(E_n) \displaystyle\prod_{ \substack{ p\: \nmid m \\ p \mid g(n)}}c_p(E_n) \displaystyle\prod_{ \substack{ p\mid m \\ p \mid g(n)}}c_p(E_n) \leq 16 \displaystyle\prod_{ \substack{ p\mid m \\ p \: \nmid g(n)}} v_p(m) \displaystyle\prod_{ \substack{ p\mid m \\ p \mid g(n)}}v_p(mg(n)), $$
where the factor $16$ comes from the primes $p=2$ and $p=3$, which may or may not be primes of additive reduction. Moreover, since $g(n)$ is square-free we see that if $p \min m$ and $p \mid g(n)$, then $v_p(mg(n))=v_p(m)+v_p(g(n))=v_p(m)+1$.
Thus, 
\begin{align*}16 \displaystyle\prod_{ \substack{ p\mid m \\ p \: \nmid g(n)}} v_p(m) \displaystyle\prod_{ \substack{ p\mid m \\ p \mid g(n)}}v_p(mg(n)) \leq 16\displaystyle\prod_{ \substack{ p\mid m \\ p \: \nmid g(n)}} (v_p(m)+1) \displaystyle\prod_{ \substack{ p\mid m \\ p \mid g(n)}}(v_p(m)+1)=16\displaystyle\prod_{  p\mid m } (v_p(m)+1)\\ \leq 16 \displaystyle\prod_{  p\mid m }(\log_2(m)+1)=16(\log_2(m)+1)^{d(m)}, \end{align*}
where $d(m)$ is the number of prime divisors of $m$.

    { \it Proof of $(ii)$:} Assume that $\Delta(T)=mg(T)$, where $g(T)$ is as in $(ii)$. According to \cite{erdos1953} there exist infinitely many positive integers $n$ for which $g(n)$ is square-free. Let $n \in \mathbb{N}$ such that $g(n)$ is square-free and consider the specialization $E_n/\mathbb{Q}$ of $E/\mathbb{Q}(T)$ at $T=n$. Proceeding exactly as in part $(i)$ we find that $c(E_n) \leq 16 (\log_2(m)+1)^{d(m)}$, where $d(m)$ is the number of prime divisors of $m$. This proves our theorem.
\end{proof}

Finally, relying on the abc-conjecture we can prove the following result.
    
\begin{theorem}\label{theoremabc}
    Let $E/\mathbb{Q}(T)$ be a non-isotrivial elliptic curve and let $\Delta(T)$ be the discriminant of a Weierstrass equation for $E/\mathbb{Q}(T)$ which is minimal at all places of $\mathbb{Q}[T]$.  Assume that $\Delta(T)$ has no repeated roots and that the largest integer that divides all elements of the set $\{ \Delta(s) \: | \: s \in \mathbb{Z}\}$ is square-free. If the abc-conjecture holds, then there exist infinitely many $n \in \mathbb{N}$ such that $c(E_n) = 1$.
\end{theorem}
\begin{proof}[Proof of Theorem \ref{theoremabc}]
    According to \cite[Theorem 1]{granvilleabc} there exist infinitely many $n \in \mathbb{Z}$ such that $\Delta(n)$ is squarefree. For each such $n$ consider the specialization $E_n/\mathbb{Q}$ of $E/\mathbb{Q}(T)$ at $T=n$. Thus we have that $v_p(\Delta(n)) \leq 1$ for all primes $p$. On the other hand, it follows from \cite[Page 365]{silverman2} and \cite[Tableau II \& Tableau IV]{pap} that if $v_p(\Delta(n)) \leq 1$, then $c_p(E_n) \leq v_p(E_n)$. Therefore, we have that $c(E_n)=\displaystyle\prod_{p \mid \Delta(n)} c_p(E_n) \leq \displaystyle\prod_{p \mid \Delta(n)}  v_p(E_n)=1.$ This proves the theorem.
\end{proof}

\bibliographystyle{plain}
\bibliography{bibliography.bib}

\end{document}